\documentclass[preprint,12pt,number]{elsarticle} 

\journal{Differential equations}

\setcitestyle{square}

\begin{document}

\begin{frontmatter}
\title{Contracting differential equations\\ in weighted Banach spaces}

\author[inst1]{Anand Srinivasan\corref{cor1}}
\ead{asrini@mit.edu}
\affiliation[inst1]{organization={Department of Mathematics, MIT},
            addressline={77 Massachusetts Avenue}, 
            city={Cambridge},
            postcode={02139}, 
            state={MA},
            country={USA}
            }

\author[inst2]{Jean-Jacques Slotine}
\ead{jjs@mit.edu}
\affiliation[inst2]{organization={Nonlinear Systems Laboratory, MIT},
            addressline={77 Massachusetts Avenue}, 
            city={Cambridge},
            postcode={02139}, 
            state={MA},
            country={USA}
            }
            
\cortext[cor1]{Corresponding author (Present affiliation: Harvard University, Cambridge, MA 02138)}

\begin{abstract}
Geodesic contraction in vector-valued differential equations is readily verified by linearized operators which are uniformly negative-definite in the Riemannian metric. In the infinite-dimensional setting, however, such analysis is generally restricted to norm-contracting systems. We develop a generalization of geodesic contraction rates to Banach spaces using a smoothly-weighted semi-inner product structure on tangent spaces. We show that negative contraction rates in bijectively weighted spaces imply asymptotic norm-contraction, and apply recent results on asymptotic contractions in Banach spaces to establish the existence of fixed points. We show that contraction in surjectively weighted spaces verify non-equilibrium asymptotic properties, such as convergence to finite- and infinite-dimensional subspaces, submanifolds, limit cycles, and phase-locking phenomena. We use contraction rates in weighted Sobolev spaces to establish existence and continuous data dependence in nonlinear PDEs, and pose a method for constructing weak solutions using vanishing one-sided Lipschitz approximations. We discuss applications to control and order reduction of PDEs.
\end{abstract}




\end{frontmatter}

\section{Introduction}

Abstract differential equations in Banach spaces arise in the study of infinite-dimensional dynamical systems, such as those governed by partial or delay differential equations. In this work, we consider the Cauchy (initial value) problem:
\begin{align}
    \label{eq:cauchy1}
    \frac{d}{dt}u &= f(t, u),\quad u(0) = u_0
\end{align}
such that $u_0$ and the domain of $f(t, \cdot)$ are contained in some subset $U$ of a Banach space $V$. We are primarily concerned with the long-time behavior of \eqref{eq:cauchy1}, which additionally leads to existence and uniqueness properties in the time-independent case. We begin with a discussion of some of the difficulties of stability analysis in the infinite-dimensional setting, and subsequently introduce our main results for establishing the incremental stability of \eqref{eq:cauchy1}.

When $V$ is finite-dimensional (e.g. if \eqref{eq:cauchy1} is an ordinary differential equation in $\C^n$), asymptotic stability about a fixed point may be thought of as an intrinsic topological property of the dynamics, as all norms on fixed-dimension vector spaces over $\R$ ($\C$) are homeomorphic and therefore preserve the asymptotic stability.

When $V$ is infinite-dimensional, on the other hand, distinct norms often induce distinct topologies -- thus, convergence to an equilibrium in one norm may not imply convergence in another. For example, consider an $\R$-indexed set of copies of an identical ordinary differential equation $\dot{u}(x) = f(u(x))$ which is globally asymptotically stable about $c \ne 0$. While the function $u(t, x)$ approaches $u(x) \equiv c$ in $L^\infty(\R)$, it is divergent in $L^p$ for $1 \le p < \infty$. An example where the stability in different norms is controlled by a nonlinearity rather than an unbounded domain is shown in \cite{coron2015dissipative}, where the authors consider the following nonlinear hyperbolic system with coupled boundaries:
\begin{align*}
    u : \R^+ \times [0, 1] &\to \R^n\\
    u_t + F(u)u_x &= 0\\
    u(t, 0) &= G(u(t, 1))
\end{align*}
In Theorems 2 and 3, \cite{coron2015dissipative}, it is shown that
there exist conditions on the coupling $G$ which result in exponential stability of the dynamics about $0$ in the Sobolev space $H^2$, yet do not imply the same in the $C^1$ norm. Thus, stability in infinite-dimensional spaces is an extrinsic property, dependent additionally upon the choice of norm.

Recent work \cite{lohmiller1998contraction, soderlind2006logarithmic, lohmiller2005pde, aminzare2013logarithmic} has sparked interest in a non-asymptotic property of dynamical systems, called \textit{incremental} stability \cite{ruffer2013convergent} -- the property that distances between any pair of solutions are monotonically decreasing in time. This is due to the relative ease of verifying incremental stability using the following ``global'' spectral analysis of the linearization \cite{lohmiller1998contraction}, compared to the construction of a Lyapunov functional. If the symmetrized Jacobian is uniformly (in time \& space) negative-definite, the system satisfies the  exponential growth bound between any pair of solutions  $u(t), v(t)$:
\begin{align}
    \label{eq:contracting}
   \exists \lambda < 0,\ \forall t \ge 0\quad d(u(t), v(t)) \le e^{\lambda t} d(u_0, v_0)
\end{align}
and is said to be \textit{contracting}.

It is known  \cite{lohmiller1998contraction} that in the automonous case, contraction implies global exponential stability in complete metric spaces, by application of the Banach fixed-point theorem. In the non-autonomous case, contraction verifies the existence of a limiting (possibly non-equilibrium) trajectory. These spectral conditions have been extended to the functional-analytic setting through the use of logarithmic norms, \cite{soderlind2006logarithmic} operator measures, \cite{cisneros2020contraction} and semi-inner products \cite{soderlind2006logarithmic, lumer1961semi, giles1967classes, aminzare2014contraction} on Hilbert and Banach spaces.

While contraction methods provide an attractive alternative to Lyapunov analysis for nonlinear systems, it is also known (see Section \ref{sec:metric_dependence} or \cite{banasiak2020logarithmic}) that contraction rates are strongly dependent upon the choice of coordinates and norm, and thus are not true dynamical invariants. A solution to this issue for finite-dimensional systems was introduced in \cite{lohmiller1998contraction}, the construction of a ``generalized Jacobian'' inducing the linearized dynamics after an arbitrary differential coordinate transform (Riemannian metric on $\R^n$) which relaxes the condition of contraction in the norm to more general geodesic distances. This was subsequently generalized to the case of arbitrary (finite-dimensional) Riemannian manifolds in \cite{simpson2014contraction, bullo2021contraction}.

The main contribution of our work is to generalize this geodesic contraction analysis to infinite-dimensional spaces. Rather than explicitly constructing Hilbert or Banach manifolds with Riemannian metrics (which may be well-approximated by open sets of the model space in the separable setting \cite{henderson1969infinite}), we use a time- and space-varying semi-inner product \cite{lumer1961semi} on tangent spaces, which we refer to as a \textit{weighted} Banach space. This gives an implicit, and non-unique, notion of geodesic distance which we measure via growth estimates on perturbations. We show, under uniform boundedness assumptions, that contraction in a bijectively weighted Banach space implies \textit{asymptotic} contraction in the norm. This controls for the possibility of initial expansive transients (see Section \ref{sec:metric_dependence}) and yields an incremental stability condition for infinite-dimensional systems which is invariant under a broad class of coordinate transforms. Quantified over possible weights within a condition number radius, we refer to this as the \textit{asymptotic} contraction rate (Definition \ref{def:as_rate}) for the system, which we show upper-bounds the maximum Lyapunov exponent.

We then apply a recently shown asymptotic contraction mapping principle for Banach spaces \cite{kirk2003fixed} to show the existence of equilibria in autonomous systems which are contracting in weighted Banach spaces. We show that the asymptotic contraction rate is a dynamical invariant up to a choice of two parameters, the uniform bound on the weight and choice of norm.

In the second part (Section \ref{sec:invariant_sets} and later) we relax the condition of an invertible weight to a merely surjective one. This produces a family of seminormed tangent spaces in which we show that negative contraction rates verify several \textit{non-equilibrium} asymptotic properties, such as convergence to subspaces (Section \ref{sec:subspace_contraction}), submanifolds (Section \ref{sec:manifold_contraction}), and invariant sets of transformations (Section \ref{sec:symmetries}). Through the use of weighting schemes, we relax the requirement that contracting phase-space partitions be convex to merely convex after an arbitrary finite transient.

Lastly, we present several applications of contraction analysis in weighted Banach spaces to partial differential equations in Section \ref{sec:ap_pdes}. We show that surjectively weighted spaces can be used to verify long-time properties of conservation laws, and that asymptotic contraction rates establish the existence and uniqueness of solutions to time-independent equations under fairly general assumptions. We show how weighted contraction analysis can be used to infer regularity properties -- in particular, introducing regularizing operators to construct weak solutions via vanishing one-sided Lipschitz approximations, in a manner similar to the vanishing viscosity method.

\subsection{Related work}

In \cite{lohmiller1998contraction}, the authors introduced Riemannian metrics on $\R^n$ to measure contraction rates of dynamics in geodesic distances other than the norm. This enabled the verification of incremental stability in vector-valued systems with various forms of interconnection, in which the choice of Riemannian metric could yield a uniformly negative-definite Jacobian. Section 3.7, \cite{lohmiller1998contraction} mentions the possibility of using weighted matrix measures (see also \cite{desoer1972measure}), which is the central topic of our development here (using semi-inner products on Banach spaces rather than matrix measures on $\R^n$). In \cite{cisneros2020contraction, soderlind2006logarithmic}, several formulations of the contraction rate in infinite-dimensional settings such as Hilbert and Banach spaces, using operator measures and semi-inner products, are introduced and reviewed. These are restricted to norm-contracting systems, except in the case of time-invariant surjectively weighted inner-product on Hilbert spaces in \cite{cisneros2020contraction}, which is referred to as \textit{partial} contraction. Our notion of geodesic contraction, on the other hand, relies on a (nonautonomous, space-varying) bijectively or surjectively-weighted semi-inner product.

In \cite{bullo2021contraction}, the authors apply weighted $\ell^1(\R^n)$ and $\ell^\infty(\R^n)$ norms to derive broad stability conditions for a class of numerical algorithms, such as recurrent neural networks \cite{davydov2021non2}. Our results use weighted contraction rates in general complete normed spaces to verify asymptotic equilibrium and non-equilibrium behaviors in infinite-dimensional and continuum systems.

In \cite{aminzare2014contraction}, the authors apply diagonally weighted norms to establish the incremental stability of diffusion and reaction-diffusion systems. We employ a more general weighting scheme, such as those induced by projection operators to invariant subspaces (\ref{thm:subspace}) and tangent spaces of invariant manifolds (\ref{thm:hbm_contraction}) to illustrate non-equilibrium properties such as convergence to subsets (\ref{thm:limit_cycle}).

\section{Main results}
\label{sec:main}
Our main construction is a weighted semi-inner product structure on a Banach space, which we introduce below.

\begin{definition}[Banach spaces \& bounded linear operators]
A normed vector space $V, \norm{\cdot}_V$ is a \textit{Banach} space if and only if every Cauchy sequence $\{u_i\}_{i=1}^\infty \subset V$ is convergent in $V$. The space of bounded linear operators $\mathcal{B}(V, W)$ between two normed spaces $V, W$ is the set of linear maps $T : V \to W$ such that 
\begin{align*}
    \norm{T}_{\mathcal{B}(V, W)} &:= \sup_{v\in V,\norm{v} = 1}\norm{Tv}_W < \infty
\end{align*}
where $\norm{\cdot}_{\mathcal{B}(V, W)}$ defines a norm on $\mathcal{B}(V, W)$. We define $\mathcal{B}(V) := \mathcal{B}(V, V)$ and drop the index $\norm{\cdot}_V$ when clear from context.
\end{definition}

\begin{definition}[Weighted Banach space]
\label{def:w_bsp}
Let $V, W$ be Banach spaces and $\Theta(t, v) \in \B(V, W)$ a family of surjective bounded linear operators parametrized by $t \in \R_+$ and $v \in V$. We define a \textit{weighted} space $V_\Theta$ to be a $t, v$-parametrized family of vector spaces with semi-norms:
\begin{align}
    \label{eq:weighted_bsp}
    \norm{u}_{\Theta(t, v)} &= \norm{\Theta(t, v)u}_W
\end{align}
where $\Theta(t, v)$ is continuously Fr\'echet-differentiable on $(0, \infty)\times V$ (henceforth referred to as $C^1$) with partial derivatives denoted by $\frac{\partial \Theta}{\partial t}, \frac{\partial \Theta}{\partial v}$, and is uniformly bounded:
\begin{align*}
    \sup_{t\in\R_+, v\in V} \norm{\Theta(t, v)}_{\B(V, W)} < \infty
\end{align*}
If additionally $\Theta(t, v)$ is injective with uniformly bounded condition number:
\begin{align*}
    \sup_{t\in\R_+, v\in V}\kappa(\Theta(t, v)) &= \sup_{t\in\R_+, v\in V}\norm{\Theta(t, v)}\norm{\Theta(t,v)^{-1}} < \infty
\end{align*}
then \eqref{eq:weighted_bsp} defines a norm, and we refer to $V_\Theta$ as a weighted Banach space.
\end{definition}

For the above family of spaces, we define an associated semi-inner product in the sense of Lumer \cite{lumer1961semi} and Giles \cite{giles1967classes}, induced by the right-handed Gateaux derivative of the weighted norm (see 5.1, \cite{soderlind2006logarithmic}).
\begin{definition}[Semi-inner product associated with a weighted Banach space]
\label{def:weighted_sip}
Let $V_\Theta$ be a weighted Banach space as in Definition \ref{def:w_bsp}. Let $t \ge 0, p, u, v \in V$, and define the $t, p$-varying semi-inner product $[\cdot, \cdot]_\Theta$ associated with $V_\Theta$ as:
\begin{align}
    \label{eq:weighted_sip}
    [u, v]_{\Theta(t, p)} := \norm{u}_{\Theta(t, p)}\lim_{h\to0^+}\frac{\norm{u + hv}_{\Theta(t, p)} - \norm{u}_{\Theta(t, p)}}{h} 
\end{align}
When $\Theta(t, p) = I$, we use the index $[\cdot, \cdot]_V$ or omit it when clear from context.
\end{definition}
In prior work (e.g. \cite{soderlind2006logarithmic, aminzare2014contraction}), it was customary to distinguish the left- and right-Gateaux derivatives of the (unweighted) norm in \eqref{eq:weighted_sip} using the notation $(\cdot, \cdot)_{\pm}$. Since we only use the right-handed derivative for forward stability analysis, we drop the index $+$ and replace it with the weight $\Theta(t, p)$.
\begin{lemma}[Properties of the semi-inner product]
Let $V_\Theta$ be a weighted Banach space, $W$ the codomain of $\Theta$, and $[\cdot, \cdot]_{\Theta}$ the associated semi-inner product. For all $u, v, w, p \in V$ and $t \ge 0$,
\begin{enumerate}
    \item $[u, v]_{\Theta(t, p)}$ exists 
    \item $[u,u]_{\Theta(t, p)} = \norm{u}_{\Theta(t, p)}^2$
    \item $|[u, v]_{\Theta(t, p)}|^2 \le \norm{u}_{\Theta(t, p)}^2\norm{v}_{\Theta(t, p)}^2$
    \item $[u, v + w]_{\Theta(t, p)} \le [u, v]_{\Theta(t, p)} + [u, w]_{\Theta(t, p)}$
    \item For all $\alpha \ge 0$,
    \begin{align*}
        \alpha[u, v]_{\Theta(t, p)} &= [\alpha u, v]_{\Theta(t, p)} = [u, \alpha v]_{\Theta(t, p)}
    \end{align*}
\end{enumerate}
\end{lemma}
\begin{proof}

\begin{enumerate}
    \item[1.] $W$ is a normed space, and the norm has left- and right-Gateaux derivatives by convexity. The limit in \eqref{eq:weighted_sip} is the right-handed Gateaux derivative of $\norm{\cdot}_W$.
    \item[2.-4.] These directly follow from the properties of the un-weighted semi-inner product $[\cdot,\cdot]_W$ on $W$, which themselves follow from the properties of $\norm{\cdot}_W$ (see e.g. equation (5.2) and Proposition 5.1 in \cite{soderlind2006logarithmic}).
    \item[5.] Seen by reparametrization $h \mapsto h\alpha$ and $h \mapsto h\alpha^{-1}$ respectively when $\alpha > 0$, and direct substitution otherwise.
\end{enumerate}
\end{proof}

Using \eqref{eq:weighted_sip}, we generalize geodesic contraction in Riemannian metrics to Banach spaces by emulating continuously-varying dual pairings on tangent spaces. Formally, we consider $V$ to be a Banach manifold with $T_pV \isoto V$ via the trivial isomorphism, and assign $T_pV$ the $C^1$-weighted semi-inner product
\begin{align}
    \label{eq:tangent_space}
    [u, v]_{\Theta(t, p)}\ \quad \forall u, v \in T_pV
\end{align}
giving a notion of nonzero ``curvature'' to $V$. Note that \eqref{eq:tangent_space} is continuous in $u, v, t, p$ if and only if the norm $\norm{\cdot}_V$ is Gateaux-differentiable (Theorem 3, \cite{giles1967classes}).

We use terms such as ``curvature'' informally, as there are several possible notions of angle and geodesic distance in a Banach space. For instance, symmetrized Gateaux derivatives of the norm (Section 3.6, \cite{balestro2017angles}) yield one possibility in uniformly convex spaces, while no canonical choice is known in general. Instead of explicitly constructing such distances, we measure the contraction rate using weighted norms of perturbation vectors on tangent spaces along a solution to \eqref{eq:cauchy1} using the norm induced by \eqref{eq:tangent_space} as follows.

\begin{definition}[Weighted contraction rate]
\label{def:weighted_rate}
Let $V_\Theta$ be a weighted Banach space and $[\cdot, \cdot]_{\Theta}$ the associated semi-inner product as in Definition \ref{def:weighted_sip}. The $\Theta$-weighted contraction rate of a linear operator $A$ with domain $D(A)$ is defined as the functional:
\begin{align}
    \label{eq:weighted_rate}
    M_{\Theta(t, u)}[A] &= \sup_{v \ne 0,\ v \in D(A)} \frac{[\Theta(t, u)v, (\frac{\partial\Theta}{\partial t}(t, u) + \Theta(t, u)A)v]_W}{\norm{v}_{\Theta(t, u)}^2}
\end{align}
Note that $A$ need not be bounded. If $\Theta$ is time-invariant, then 
\begin{align*}
    M_{\Theta(u)}[A] &= \sup_{v \ne 0, v \in D(A)}\frac{[v, Av]_{\Theta(u)}}{\norm{v}^2_{\Theta(u)}}
\end{align*}
\end{definition}

We show that $M_{\Theta(t, v)}$ verifies incremental stability of the dynamics \eqref{eq:cauchy1} in the differential sense below. In the following and throughout, we use the shorthand $Df_t(u) := \frac{\partial f}{\partial u}(t, u)$.

\begin{theorem}[Growth of perturbations in weighted spaces]
\label{thm:growth}
Let $V_\Theta$ be a weighted Banach space and $f(t, u) \in C^1(\R_+ \times V)$. If 
\begin{align*}
    \sup_{u \in V}M_{\Theta(t, u)}[Df_t(u)] \le \lambda(t)\ \quad \forall t \ge 0
\end{align*}
then any perturbation $\delta u(t) \in T_{u(t)}V$ at a point $u(t)$ evolving by the time-varying linear dynamics:
\begin{align*}
    \frac{d}{dt}\delta u &= Df_t(u)\delta u
\end{align*}
satisfies the growth bound:
\begin{align*}
    \norm{\delta u(t)}_{\Theta(t, u)} &\le \exp\left(\int_{t_0}^t \lambda(s)ds\right) \norm{\delta u(t_0)}_{\Theta(t_0, u)}
\end{align*}
for all $t \ge t_0 \ge 0$.
\end{theorem}
\begin{proof}
Let $d^+_t$ denote the right-handed Gateaux derivative of a scalar function of $t$. Note that $d^+_t\norm{\delta u(t)}$ exists finite for any norm by convexity. By the chain rule for Gateaux derivatives, 
\begin{align*}
    d^+_t\norm{\delta u(t)}_{\Theta(t, u)} &=\lim_{h\to0^+} \frac{\norm{\Theta(t, u)\delta u(t) + h\frac{d}{dt}[\Theta(t, u)\delta u(t)]}_W - \norm{\Theta(t, u)\delta u(t)}_W}{h}\\
    &= [\Theta(t, u)\delta u(t), \frac{d}{dt}[\Theta(t, u)\delta u(t)]]_W\\
    &\le M_{\Theta(t, u)}[Df_t(u)]\norm{\delta u(t)}_{\Theta(t, u)}\\
    &\le \lambda(t) \norm{\delta u(t)}_{\Theta(t, u)}
\end{align*}
Since this holds for arbitrary $u \in V$, the result follows by an application of Gr\"onwall's inequality for right-handed derivatives.
\end{proof}

We now show that contraction in a suitably weighted space implies asymptotic contraction in the unweighted space. Furthermore, asymptotic contraction rates are invariant under the choice of weight. This result is also novel in finite dimension, although in that case the topological equivalence of uniformly bounded Riemannian metrics gives the result immediately.

\begin{definition}[Asymptotically contracting dynamics]
\label{def:ac}
Let $\phi(t, u_0) : \R_+ \times W$ be the propagator induced by solutions to the Cauchy problem \eqref{eq:cauchy1}:
\begin{align*}
    \frac{d}{dt}u(t) &= f(t, u)\ \forall t\\
    u(0) &= u_0\\
    \phi(t, u_0) &:= u(t)
\end{align*}
for an initial condition $u_0 \in W$ with $W$ a closed convex subset of a Banach space $V$. We say that \eqref{eq:cauchy1} is asymptotically contracting if there exists $\lambda_0 \in [0, 1)$ and a function $\lambda(t) : \R_+ \to [0, \infty)$ with $\lambda(t) \to \lambda_0$ as $t \to \infty$ such that:
\begin{align*}
    \norm{\phi(t, u_0) - \phi(t, v_0)} \le \lambda(t)\norm{u_0 - v_0}
\end{align*}
for all $t \ge 0$ and $u, v \in W$.
\end{definition}
When $f$ is time-invariant and there exists $u_0 \in W$ such that $\sup_t\norm{\phi(t, u_0)} < \infty$ (a bounded solution), Theorem 2.1, \cite{kirk2003fixed} establishes that asymptotically contracting dynamics approach fixed points.

\begin{theorem}[Contraction in an invertibly weighted space]
\label{thm:c_cont}
Let $\Theta(t, u) \in C^1(\R_+ \times V, GL(V, W))$ be a continuously differentiable family of bijective bounded linear operators between Banach spaces $V, W$ with the uniform boundedness property:
\begin{align*}
    \norm{\Theta(t, u)} \le B_1, \norm{\Theta(t, u)^{-1}} \le B_2\quad \forall t \ge 0,\ u \in V
\end{align*}
and $M_{\Theta(t, u)}$ the $\Theta$-weighted contraction rate as in Definition \ref{def:weighted_rate}. Let $u(t)$ be a solution of \eqref{eq:cauchy1} and suppose that:
\begin{align*}
    \sup_{t\ge 0,\ u\in V}M_{\Theta(t, u)}[Df_t(u)] \le \lambda
\end{align*}
Then:
\begin{enumerate}
    \item $M_{\Theta(t, u)}[Df(t,u)] = M_W[(\frac{\partial\Theta}{\partial t}(t, u) + \Theta(t,u) Df_t(u)) \Theta(t,u)^{-1}]$
    \item For any two solutions $u_1(t), u_2(t)$ to \eqref{eq:cauchy1}, we have:
    $$
    \norm{u_1(t) - u_2(t)} \le B_1B_2e^{\lambda t}\norm{u_1(0)-u_2(0)}
    $$
    with $B_1B_2 = \kappa(\Theta)$ the condition number for constant $\Theta$.
    \item If $\lambda < 0$, then \eqref{eq:cauchy1} is an asymptotic contraction in $V$ with rate $\lambda$.
    \item The asymptotic contraction rate $\lambda$ is invariant under choice of $\Theta(t, u)$.
\end{enumerate}
\end{theorem}
\begin{proof}
(1) is seen by reparametrization $w = \Theta(t, u)v$:
\begin{align*}
    M_{\Theta(t,u)}[Df_t(u)] &= \sup_{v \ne 0\in V}\frac{[\Theta(t,u)v, (\frac{\partial \Theta}{\partial t}(t,u) + \Theta(t,u)Df_t(u))v]}{\norm{\Theta(t, u)v}_V^2}\\ 
    &= \sup_{w\ne 0 \in W}\frac{[w, (\frac{\partial\Theta}{\partial t}(t,u) + \Theta(t,u)Df_t(u))\Theta(t,u)^{-1})w]}{\norm{w}_W^2}\\ 
    &= M_W\left[\left(\frac{\partial\Theta}{\partial t}(t, u) + \Theta(t,u) Df_t(u)\right) \Theta(t,u)^{-1}\right]
\end{align*} 
In $\R^n$, the argument of $M_W$ in the last line in is referred to in \cite{lohmiller1998contraction} as the \textit{generalized Jacobian}.

We show (2) using linearity of $\Theta$ and (1). Let $\delta u(t)$ a perturbation as in Theorem \ref{thm:growth} and $\delta v = \Theta(t, u)\delta u$ its conjugate. By hypothesis and Theorem \ref{thm:growth},
\begin{align*}
    \norm{\delta v(t)}_W &\le e^{\lambda t}\norm{\delta v(0)}_W \implies\\
    \norm{\delta u(t)}_V &\le \norm{\Theta(t, u)^{-1}}\norm{\Theta(0, u)}\norm{\delta u(0)}e^{\lambda t} \le B_1B_2 \norm{\delta u(0)}e^{\lambda t}
\end{align*}
Applying the fundamental theorem of calculus for Fr\'echet spaces, condition (2) follows.

This immediately implies (3), since if $\lambda < 0$ then \eqref{eq:cauchy1} induces a contraction map after a transient of length $t_0 = -\lambda^{-1}\log(B_1B_2)$.

Finally, (3) implies (4) as follows. If $\Theta_a, \Theta_b$ are two weights satisfying the hypothesis with constants $B_1^a, B_2^a, B_1^b, B_2^b$ and associated codomains $W_a, W_b$, then 
\begin{align*}
    &\norm{\delta u(t)}_{\Theta_a(t,u)} \le e^{\lambda t}\norm{\delta u(0)}_{\Theta_a(0,u)} \implies\\
    &\norm{\delta u(t)}_{V} \le B_1^aB_2^ae^{\lambda t}\norm{\delta u(0)}_{V} \implies\\
    &\norm{\delta u(t)}_{\Theta_b(t,u)} \le B_1^aB_2^aB_1^bB_2^be^{\lambda t}\norm{\delta u(0)}_{\Theta_b(0,u)}
\end{align*}
Hence the \textit{asymptotic} contraction rate (defined as the existence of a finite prefactor such that the growth bound holds) is invariant under $C^1$, uniformly bounded, invertible weights, giving statement (4).
\end{proof}

Theorem \ref{thm:c_cont} establishes the (differential) coordinate independence of asymptotic contraction rates. Accordingly, we define the \textit{asymptotic} contraction rate within a uniformly bounded family of weights as follows. 
\begin{definition}[Radius-$b$ family of weights]
Let $b > 0$ and $V$ be a Banach space. We define the $b$-uniformly bounded, $C^1$ family of invertible maps into a Banach space $W$ as:
\begin{equation}
\begin{split}
    \label{eq:unif_weight}
    B_b(W) = \{&\Theta(t, u) \in C^1(\R_+ \times V, GL(V, W))\ |\\
    &\sup_{t\ge0, u\in V}\max(\norm{\Theta(t, u)},\norm{\Theta(t,u)^{-1}}) \le b\}
\end{split}
\end{equation}
where $GL(V, W) \subset \mathcal{B}(V, W)$ is defined as the space of bijective bounded linear operators from $V$ to $W$.
\end{definition}
For each $b > 0$, $B_b(W)$ defines an admissible family of weights per Definition \ref{def:w_bsp}, which we refer to as having \textit{radius} $b$.

\begin{definition}[Asymptotic contraction rate]
\label{def:as_rate}
Let $b > 0$, $V, W$ be Banach spaces and $B_b(W)$ be a radius-$b$ family of weights. We define the asymptotic contraction rate (parametrized by $b$) of the Cauchy problem \eqref{eq:cauchy1} as:
\begin{align}
    \label{eq:as_rate}
    \lambda_{b} &= \inf_{\Theta(t,u) \in \mathcal{B}_b(W)}\sup_{t \ge 0,\ u\in V}M_{\Theta(t,u)}[Df_t(u)]
\end{align}
\end{definition}
By Theorem \ref{thm:c_cont}, any dynamics with asymptotic contraction rate $\lambda_b < 0$ is contracting in the norm $\norm{\cdot}_V$ after a transient of length at most $t_b = -2\lambda_b^{-1}\log(b)$. We note that unlike the classical contraction mapping principle, a connected region $U \subseteq V$ in which \eqref{eq:cauchy1} is asymptotically contracting (has $\lambda_b < 0$) need not be geodesically convex; in a sense, $U$ need only be \textit{asymptotically} convex, that is, convex under a $t_b$-length application of the propagator of \eqref{eq:cauchy1}.

Lastly, as a shorthand we will define $M_\Theta[f]$ for nonlinear, nonautonomous, continuously differentiable $f$ by collapsing the rightmost term of \eqref{eq:as_rate} as:
\begin{align}
    \label{eq:m_nonlinear}
    M_\Theta[f] &:= \sup_{t \ge 0, u \in V}M_{\Theta(t, u)}[Df_t(u)]
\end{align}
This quantity upper-bounds, but is distinct from, the \textit{one-sided Lipschitz constant} of $f$ \cite{davydov2021non}. We use \eqref{eq:m_nonlinear} instead of $osL(f)$ throughout in order to employ the weighting scheme \eqref{eq:weighted_rate}.

\subsection{Coordinate invariance of asymptotic contraction rates}
\label{sec:metric_dependence}
Theorem \ref{thm:c_cont} establishes the invariance of asymptotic, rather than time-uniform, contraction under differential coordinate changes. Definition \ref{def:as_rate} makes clear that the asymptotic contraction rate lower-bounds the particularly weighted contraction rates, thus is an intrinsic property of \eqref{eq:cauchy1} up to the choice of norm and bound $b$. We show that this property is analogous to, but more general than, the invariance of Lyapunov exponents for autonomous dynamics under measure-preserving invertible transformations.

Let $\delta u_0$ be an initial perturbation of an autonomous dynamics \eqref{eq:cauchy1} satisfying the conditions of Oseledets' theorem \cite{ruelle1979ergodic}; the maximum Lyapunov exponent $\lambda_{MLE}$ is upper-bounded by the (unweighted) contraction rate $\lambda$, by Theorem \ref{thm:growth}:
\begin{align}
    \label{eq:mle}
    \lambda_{MLE} &= \lim_{t\to\infty}\lim_{\norm{\delta u_0}\to0}\frac1t\log \frac{\norm{\delta u(t)}}{\norm{\delta u_0}} \le \lim_{t\to\infty}\lim_{\norm{\delta u_0}\to0}\frac1t\log \frac{\norm{\delta u_0}e^{\lambda t}}{\norm{\delta u_0}} = \lambda
\end{align}
suggesting that $\lambda$ may be used as an estimator for the Lyapunov exponent. An alternate way to establish \eqref{eq:mle} for bounded linearizations $Df_t(u)$ is via the logarithmic norm \cite{dahlquist1958stability} (or \textit{matrix measure} \cite{desoer1972measure}) $\mu$:
\begin{align*}
    D_t^+\log \norm{\delta u(t)} &= \frac{d_t^+ \norm{\delta u(t)}}{\norm{\delta u(t)}} \le \mu(Df_t(u)) \le \lambda
\end{align*}
where $D_t^+$ is the upper-right Dini derivative. The last inequality holds due to the identity $\mu(A) = M[A]$ for bounded linear $A$ (Lemma 12, \cite{lumer1961semi}).

We make some observations about the inequality \eqref{eq:mle}. First, since $\lambda_{MLE}$ measures an asymptotic property while the contraction rate $\lambda$ measures a (time-)uniform property of the dynamics, $\lambda_{MLE}$ is norm-independent in finite dimension while $\lambda$ is not. Thus, we expect $\lambda$ to be larger, accounting for expansive transient periods. Moreover, $\lambda_{MLE}$ is invariant under measure-preserving transformations, and thus isometries in the case of the Lebesgue measure-induced metric on $\R^n$, while $\lambda$ is only invariant under isometry. However, even in this case $\lambda$ naturally extends to non-autonomous dynamics, while $\lambda_{MLE}$ is classically formulated for autonomous vector fields.

However, the radius-$b$ asymptotic contraction rate $\lambda_{b}$ is invariant not only under (Lipschitz) invertible phase-space transformations, but also under arbitrary $b$-bounded differential coordinate changes $\delta v = \Theta \delta u$ for which a phase-space coordinate change $v = f(u)$ may not necessarily exist unless $\Theta = Df$ (as noted in \cite{lohmiller1998contraction}, an integrability condition on $\Theta$). Moreover, it also bounds $\lambda_{MLE}$ by Theorem \ref{thm:c_cont} and substitution into \eqref{eq:mle} as:
\begin{align*}
    \lambda_{MLE} &\le b^2\lambda_b
\end{align*}
We note that the sign is preserved with $b$ arbitrarily large.
Thus, for the aim of verifying incremental stability, the asymptotic contraction rate $\lambda_b$ may be considered a more general dynamical invariant than the maximum Lyapunov exponent.

This invariance is restricted for both quantities, in the infinite-dimensional case, up to a choice of norm (Definition \ref{def:w_bsp}). In general, asymptotic contraction rates in different norms are not comparable, unless additional boundedness assumptions are imposed -- e.g. $L^p(\Omega)$ with $\Omega$ of finite measure, where H\"older's inequality yields growth estimates from $L^q$ for $1 \le p \le q$. For instance, if $\delta u(t)$ is a perturbation at a point $u(t)$ as in Theorem \ref{thm:growth} and $\lambda_q$ is the contraction rate in $\norm{\cdot}_q$, then:
\begin{align*}
    \norm{\delta u(t)}_p &\le |\Omega|^{\frac1p - \frac1q}\norm{\delta u(t)}_q \le |\Omega|^{\frac1p - \frac1q}e^{\lambda_qt}\norm{\delta u(0)}_q
\end{align*}
where $|\Omega|$ is the Lebesgue measure of $\Omega$. Consequently, restricting \eqref{eq:cauchy1} to the Banach space $L^p(\Omega) \cap L^q(\Omega)$ with norm $\norm{\cdot}_{p,q} = \max(\norm{\cdot}_p, \norm{\cdot}_q)$, we have:
\begin{align}
    \norm{\delta u(t)}_{p, q} &\le \max(1, |\Omega|^{\frac1p - \frac1q})e^{\lambda_qt}\norm{\delta u(0)}_{p, q}
\end{align}

For a further discussion of the norm-dependence of contraction rates in infinite dimension, see 3.5-3.6 \cite{banasiak2020logarithmic}.

\subsection{Contraction to invariant sets}
\label{sec:invariant_sets}

In Theorem \ref{thm:c_cont}, we used bijectively weighted spaces to derive asymptotic contraction rates for dynamics in Banach spaces. We now relax this condition to surjectively weighted spaces, showing that contraction in the induced seminorms establishes asymptotic convergence to subsets other than point-equilibria (or limiting trajectories in the nonautonomous case). 

\label{sec:subspace_contraction}
\begin{theorem}[Contraction to an invariant subspace]
\label{thm:subspace}
Let $V$ be a Banach space and $P$ a bounded linear projection on $V$ (i.e. $P^2 = P$). Let $Q = I - P$, and $M_Q[A]$ be the $Q$-weighted contraction rate. Let $\frac{d}{dt}u = f(t, u)$ and suppose that:
\begin{enumerate}
    \item $\im(P)$ is a flow-invariant subspace of $u(t)$
    \item $\sup_{t, u \in V}M_Q[Df(t, u)] = \lambda < 0$
\end{enumerate}
Then, $u(t)$ is contracting to the subspace $\im(P)$:
\begin{align*}
    \inf_{v\in \im(P)}\norm{u(t) - v} &\le e^{\lambda t}\inf_{v\in \im(P)}\norm{u_0 - v}
\end{align*}
\end{theorem}
\begin{proof}
$P$ gives a decomposition of $V$ into a direct sum of closed subspaces as $V = \im(P) \oplus \ker(P)$; moreover, $Q$ is also a projection, and satisfies $\im(Q) = \ker(P)$. Consider the projected system $v = Qu$ with dynamics 
\begin{align}
    \label{eq:subspace_dynamics}
    \frac{d}{dt}v &= Qf(t, u) = Qf(t, v + Pu)
\end{align}
By hypothesis:
\begin{enumerate}
    \item Since $\im(P)$ is a flow-invariant subspace of $u(t)$, we have
    \begin{align}
        \label{eq:invariant_subspace}
        Qf(t, Pv) = 0\ \forall t, v \in V
    \end{align}
    thus $v(t)$ has a fixed point at $v = 0$. 
    \item Since $u(t)$ is contracting in the $Q$-weighted space, we have
    \begin{equation}
    \begin{split}
        \label{eq:subspace_contraction}
        M_Q[f]  &= \sup_{t, u, w \in V} \frac{\left[Qw, QDf_t(u)w\right]}{\norm{Qw}^2} < 0
    \end{split}
    \end{equation}
    Then by linearity, $\delta v = Q\delta u$, and $v(t)$ is contracting; that is, a perturbation $\delta v$ with dynamics
    \begin{align}
    \label{eq:conj2}
    \frac{d}{dt}\delta v = Q Df_t(u)\delta v
    \end{align}
    is exponentially stable about $0$.
\end{enumerate}
Thus, all solutions $v(t)$ of \eqref{eq:subspace_dynamics} approach the fixed point $0$ exponentially, and $u(t)$ is contracting to the subspace $\im(P)$.
\end{proof} 
Note that in the $L^2$ case, if $Q = A^*A$ such that $AA^* = I$, then \eqref{eq:subspace_contraction} reduces to: 
\begin{align*}
    M_Q[Df_t(u)] &= \sup_{w \in V} \frac{\Re \langle Qw, QDf_t(u)w\rangle}{\langle Qw, Qw\rangle} = \sup \Re[\Spec(QDf_t(u)Q^*)]
\end{align*}
which reduces to the prior known condition in $\R^n$ given by \cite{pham2007stable} $QDf_t(u)Q^\intercal < 0$. On the other hand, result \eqref{eq:subspace_contraction} does not require a notion of orthogonality or self-duality.

We can further relax condition \eqref{eq:subspace_contraction} to an asymptotic contraction. Let $\Theta(t, v) : C^1(\R^+\times V, GL(V))$ be a uniformly bounded family of operators as in Theorem \ref{thm:c_cont}; if the dynamics are contracting in the inner-$\Theta$, outer-Q-weighted norm:
\begin{align}
    \label{eq:weighted_subspace_contraction}
    \sup_{t, u\in V} M_{Q\Theta(t,u)}[Df_t(u)] < 0
\end{align}
then by Theorem \ref{thm:c_cont}, \eqref{eq:conj2} is asymptotically contracting to $0$, and $u(t)$ is asymptotically contracting to the subspace $\im(P)$.

\subsection{Contraction to submanifolds}
\label{sec:manifold_contraction}
We now give sufficient conditions for contraction to Hilbert and Banach submanifolds. 

As calculating the (weighted) contraction rate is often easier in $L^2$ (see Appendix for $L^p$ contraction estimates from $L^2$); the below result along with a suitable embedding of a Hilbert submanifold $\mathcal{M}$ may be used to obtain asymptotic convergence to submanifolds in $L^p$. We subsequently give the general result for Banach submanifolds.

\begin{theorem}[Contraction to an invariant Hilbert submanifold]
\label{thm:hbm_contraction}
Let $\phi \in C^\infty(V, W)$ be a smooth submersion of a Hilbert space $V$ into another $W$ (i.e. for all $u \in V$, $D\phi(u)$ is surjective \cite{lang2012differential}).
$\mathcal{M}  = \phi^{-1}(0)$ defines a $\dim(W)$-submanifold of $V$. 

Define the surjective weight $\Theta(u) = D\phi(u)$ and let $M_{\Theta(u)}$ be the weighted contraction rate in $L^2$ (as in \ref{def:weighted_rate}):
\begin{align}
    \label{eq:manifold_weighted_ip}
    M_{\Theta(u)}[A] &= \sup_{v \ne 0, v \in V}\frac{\Re\langle\Theta(u)v, \Theta(u)Av\rangle_W}{\norm{\Theta(u)v}_W^2}
\end{align}

Let $\frac{d}{dt}u = f(t, u)$ be a differentiable vector field on $V$, and suppose that:
\begin{enumerate}
    \item $\forall u \in \mathcal{M}, t \ge 0$, we have $D\phi(u)f(t, u) = 0$
    \item $u(t)$ is differentially contracting in the $\phi$-weighted inner product \eqref{eq:manifold_weighted_ip}:
    $$
    \sup_{t \ge 0, u \in V}M_{\Theta(u)}[Df_t(u)] = \lambda < 0
    $$
\end{enumerate}
Then $u(t)$ asymptotically approaches the manifold $\mathcal{M}$.
\end{theorem}
\begin{proof}
First, we note that since $\phi$ is a submersion, quantity \eqref{eq:manifold_weighted_ip} is well-defined. We show that $u(t)\to \mathcal{M}$ by showing that $\phi(u(t))$ is contracting to $0$. By hypotheses, 
\begin{enumerate}
    \item $\phi(u(t))$ has a fixed point at $0$, by either of the following equivalent conditions (compare to the linear version \eqref{eq:subspace_dynamics}):
    \begin{align}
    \label{eq:invariant_manifold}
    \forall u \in \mathcal{M}, t \ge 0,\quad 
    \begin{cases}
        \frac{d}{dt}\phi(u) = D\phi(u)f(t, u) = 0\\
        f(t, u) \in T_u\mathcal{M}  = \ker(D\phi(u)) 
    \end{cases}
    \end{align}
    This states that the dynamics are tangent to $\mathcal{M}$, i.e. $\mathcal{M}$ is an invariant manifold of $v(t)$.

    \item Let $v = \phi(u)$; a perturbation $\delta u \in T_u V$ propagates as $\delta v = D\phi(u)\delta u$. Then, 
    \begin{align}
        \label{eq:manifold_projection}
        \frac{d}{dt}\frac12\norm{\delta v}_W^2 &= \Re \left\langle D\phi(u)\delta u, D\phi(u)Df_t(u)\delta u\right\rangle_W
    \end{align}
    
     By substitution of \eqref{eq:manifold_weighted_ip} into \eqref{eq:manifold_projection}, 
    \begin{equation}
    \begin{split}
        \label{eq:manifold_contraction}
        \frac{d}{dt}\norm{\delta v}^2 &= \frac{2\Re \langle\delta u, Df_t(u)\delta u\rangle_{\phi}}{\langle \delta u, \delta u\rangle_{\phi}}\norm{\delta v}^2 \le 2M^{\phi}(Df_t(u)) \norm{\delta v}^2\\
        &\le 2\lambda \norm{\delta v}^2
    \end{split}
    \end{equation}
\end{enumerate}

Together, \eqref{eq:invariant_manifold} and \eqref{eq:manifold_contraction} imply that $\phi(u(t)) \to 0$ exponentially and thus $u(t)$ is contracting to the submanifold $\mathcal{M} = \phi^{-1}(0)$.
\end{proof}

This can be generalized to Banach submanifolds by use of the semi-inner product as in Theorem \ref{thm:growth}. 

\begin{theorem}[Contraction to an invariant Banach submanifold]
\label{thm:banach_submanifold_contraction}
Let $V, W$ be Banach spaces and $\phi : V \to W$ define a $\dim(W)$-submanifold as in \ref{thm:hbm_contraction}. Let $\Theta = D\phi$, and correspondingly define the $\Theta$-weighted contraction rate as in Definition \ref{def:weighted_rate}:
\begin{align}
    \label{eq:bsp_weighted_norm}
    M_{\Theta(u)}[A] &= \sup_{v \ne 0, v \in V}\frac{[D\phi(u)v, D\phi(u)Av]}{\norm{D\phi(u)v}_W^2}
\end{align}
Let $\frac{d}{dt}{u} = f(t,u)$ be a Cauchy problem in $V$, and suppose that:
\begin{enumerate}
    \item $\forall u \in M, t \ge 0$, $D\phi(u)f(t, u) = 0$
    \item $u(t)$ is contracting in the $\phi$-weighted \textit{semi}-inner product \eqref{eq:bsp_weighted_norm}:
    $$
    \sup_{t \ge 0, u \in V}M_{\Theta(u)}[Df_t(u)] \le \lambda < 0
    $$
\end{enumerate}
Then $u(t) \to \mathcal{M} = \phi^{-1}(0)$ asymptotically.
\end{theorem}

\begin{proof}
Let $u(t)$ be a solution, $v(t) = \phi(u(t))$, and $\delta u, \delta v$ be perturbations as before. Taking the Gateaux derivative of the norm of a perturbation $\delta v(t)$ and applying Theorem \ref{thm:growth} and hypothesis (2),
\begin{equation}
\begin{split}
    \label{eq:banach_manifold_contraction}
    d^+_t \norm{\delta v} &= \frac{\left[\delta v, \frac{d}{dt}\delta v\right]}{\norm{\delta v}^2} \norm{\delta v} = \frac{\left[D\phi(u) \delta u, D\phi(u)Df_t(u)\delta u\right]}{\norm{D\phi(u)\delta u}^2} \norm{\delta v} \\
    & \le M_{\Theta(u)}[Df_t(u)]\norm{\delta v} \le \lambda \norm{\delta v}
\end{split}
\end{equation}
Thus, $v(t)$ is contracting. By hypothesis (1), $\mathcal{M}$ is an invariant manifold and hence $v(t) = 0$ is a particular solution. Therefore, $v(t) = \phi(u(t)) \to 0$  and $u(t)$ asymptotically converges to the submanifold $\mathcal{M} = \phi^{-1}(0)$.
\end{proof}

Condition \eqref{eq:banach_manifold_contraction} can also be extended using weighted spaces \eqref{eq:weighted_sip}. We can apply an \textit{outer} (time-invariant) weight $\Theta(w)$, for $w = v(t)$, to the $\delta v$ dynamics, and define the corresponding weighted contraction rate:
\begin{align*}
    M_{\Theta(w)D\phi(u)}[A] &= \sup_{v \ne 0, v \in V}\frac{[\Theta(w) D\phi(u) v, \Theta(w)D\phi(u)Av]}{\norm{\Theta(w)D\phi(u)v}^2} 
\end{align*}
Then if $M_{\Theta(w)D\phi(u)}[A] < 0$ uniformly, $v(t)$ is asymptotically contracting to $0$ by Theorem \ref{thm:c_cont} and the fact that $v(t) = 0$ is a particular solution. Thus, $u(t)$ is asymptotically contracting to $\mathcal{M}$.

We note that both in the cases of contraction to subspaces \eqref{eq:subspace_contraction} and submanifolds \eqref{eq:banach_manifold_contraction}, the sufficient condition is a negative contraction rate in a surjectively weighted space; the difference between the ``linear'' and ``nonlinear'' setting is a constant versus a space-varying weight.




\section{Application: convergence to symmetric solutions}
\label{sec:symmetries}

We now consider integral curves of vector fields which are contracting in some weighted Banach space and are simultaneously equivariant with respect to a group action. We show that limiting trajectories are then group-invariant. These notions were developed by \cite{russo2011symmetries} for $\R^n$, which we extend using weighted semi-inner products to dynamics in Banach spaces.

\subsection{Spatial symmetries}
\label{sec:spatial_linear}
\begin{theorem}
Let $\frac{d}{dt}u = f(t, u)$ in a Banach space $V$, and $\Gamma \subseteq GL(V)$ a subgroup of bijective bounded linear operators on $V$ such that $f$ is $\Gamma$-equivariant in space:
\begin{align}
\label{eq:spatial1}
\forall T \in \Gamma, u \in V, t \ge 0,\quad f(t, Tu) = Tf(t, u)
\end{align}
If $f$ is contracting with respect to some weighted semi-inner product $M^\Theta$, then $u(t)$ approaches a single $\Gamma$-invariant vector exponentially fast.
\end{theorem}
\begin{proof}
Define the set of solutions at time $t$ as $S(t) = \{\phi(t, u)\ |\ u \in V\}$, where $\phi$ is the propagator generating solutions to \eqref{eq:cauchy1} as in Definition \ref{def:ac}. By the equivariance \eqref{eq:spatial1}, we have
$$
\frac{d}{dt}u = f(t, u) \implies T\frac{d}{dt}u = Tf(t, u) = f(t, Tu)
$$
implying that $u \in S(t) \implies Tu \in S(t)$, thus $S(t)$ is $\Gamma$-invariant for all $t$. 

By the contraction hypothesis, $S(t)$ is then contracting after a finite overshoot to a single solution $u_*(t)$ which is $\Gamma$-invariant. 
\end{proof}

\subsubsection{Nonlinear symmetries}
\label{sec:nonlinear_symmetries}
The previous result extends straightforwardly to non-linear symmetries.

Suppose $\Gamma \subset C^1(X)$ is more generally a group of diffeomorphisms, such that $f$ is \textit{differentially} equivariant:
\begin{align}
\label{eq:nonlinear_equivariance}
\forall h \in \Gamma, u \in V, t \ge 0,\quad f(t, h(u)) = Dh(u)f(t, u)
\end{align}
(We note that this equivariance condition \eqref{eq:nonlinear_equivariance} also appears for $\R^n$ in an un-published manuscript by \cite{unpub_boffi}.)
Then, for any $h \in \Gamma$ we have 
$$
\frac{d}{dt}h(u) = Dh(u)\frac{d}{dt}u = f(t, h(u))
$$
Thus $S(t)$ is $\Gamma$-invariant. Finally, if $M_\Theta[f] < 0$ for some weight $\Theta$, 
$S(t)$ approaches a single solution $u_*(t)$ which is a fixed point of each $h \in \Gamma$.

\subsubsection{Contraction to a $\Gamma$-invariant subspace} 
Suppose we only wish to show that $u(t)$ approaches \textit{some} $\Gamma$-invariant solution rather than a unique one; we apply contraction in a weighted semi-inner product to do so.
\begin{theorem}
Let $\Gamma \subseteq GL(V)$ and $\frac{d}{dt}u = f(t, u)$ be $\Gamma$-equivariant as in \ref{sec:spatial_linear}. Define the $\Gamma$-invariant subspace:
\begin{align*}
    V_\Gamma = \{u \in V\ |\ Tu = u\ \forall T \in \Gamma\}
\end{align*}
and suppose that $\dim(V_\Gamma) < \infty$. Then, there exists a weighted space $V_\Theta$ such that if $M_\Theta[f] < 0$, solutions $u(t)$ asymptotically approach the $\Gamma$-invariant subspace $V_\Gamma$.
\end{theorem}
\begin{proof}
By the Hahn-Banach theorem, $\dim(V_\Gamma) < \infty)$ guarantees the existence of a bounded linear projection operator $P$ with $\im(P) = V_\Gamma$ (see \cite{randrianantoanina2001norm}). By $\Gamma$-equivariance, we have that:
\begin{align*}
    \forall T \in \Gamma, u \in V, t \ge 0,\quad Tf(t, Pu) &= f(t, TPu) = F(t, Pu)
\end{align*}
implying that $QF(t, Pu) = 0$ for all $t,u$ (where $Q = I - P$). Thus $V_\Gamma$ is a flow-invariant subspace of $u(t)$. 

Finally, applying Theorem \ref{thm:subspace}, contraction in the weighted space $\Theta(t, u) = Q$, with $M_{Q}[f] < 0$ implies that $u(t)$ is contracting to $V_\Gamma$. 
\end{proof}

The above result also extends to asymptotic contractions to $V_\Gamma$ by further pre-multiplicative weights on $Qu(t)$ as in Equation \eqref{eq:weighted_subspace_contraction}.

\begin{example}[Periodic heat equation]
Let $\partial_{t} u = \Delta u$ on the torus $\T^n = \R^n/\Z^n$, and $\Gamma$ be a group of bounded linear translation operators. We have $\Delta T = T\Delta$ for all $T \in \Gamma$, and $\mu(\Delta) < 0$ in $L^2_0(\T^n)$, the set of mass-zero square-integrable functions (see Proposition 6.1, \cite{soderlind2006logarithmic} for an estimation of the contraction rate using Sobolev inequalities); thus $u$ tends to a $\Gamma$-invariant solution. Since $\Gamma$ was arbitrary, this is an alternate proof that harmonics on $\T^n$ are translation-invariant.  
\end{example}




\subsection{Temporal symmetries}
\begin{theorem}
\label{thm:spatiotemporal}
Let $\frac{d}{dt}u = f(t, u)$ and $\tau > 0$ such that $f$ is $\tau$-periodic:
\begin{align}
\label{eq:fperiodic}
    \forall t \ge 0,\ u \in V,\quad f(t, u) &= f(t + \tau, u)
\end{align}
If $f$ is contracting in some weighted norm,
\begin{align}
\label{eq:spatiotemporal_contraction}
M_\Theta[f] &\le \lambda < 0
\end{align}
then solutions $u(t)$ exponentially approach a $\tau$-periodic function.
\end{theorem}
\begin{proof}
By the symmetry \eqref{eq:fperiodic}, if $u(t)$ is a solution, then for $t \ge 0$,
$$
\frac{d}{dt}u(t+\tau) = f(t + \tau, u(t+\tau)) = f(t, u(t+\tau))
$$
thus $u(t+\tau)$ is also a solution, and $S(t)$ is $\tau$-translation invariant (letting $S(t)$ be the set of solutions at time $t$ as in \ref{sec:spatial_linear}).

By the contraction hypothesis and Theorem \ref{thm:c_cont}, $S(t)$ exponentially approaches a single solution. That is, there exists $C > 0$  such that for any solution $u(t)$ and $m, n \in \N$ we have:
$$
\norm{u(t + m\tau) - u(t + n\tau)} \le Ce^{\lambda \min(m, n)\tau}\norm{u(t + \tau) - u(t)}
$$
thus $\{u(t + n\tau)\}_n$ is a Cauchy sequence. Since $V$ is a Banach space and thus complete, this implies $\lim_{n\to\infty}u(t + n\tau) = p$ for some $p \in V$. As this is true for arbitrary $t$, we have that $u(t)$ approaches a $\tau$-periodic function, which is approached exponentially by all initial conditions.
\end{proof}

\textbf{Remark.} The $\tau$-periodicity of the vector field may be shown using spatio-temporal symmetries of $f$, e.g.:
\begin{itemize}
    \item (Linear symmetry) A cyclic subgroup $\Gamma \subseteq GL(V)$ of order $k$ generated by $T \in GL(V)$ such that $T^k = I$, and $\Delta t > 0$ such that:
    \begin{align*}
        u \in V, t \ge 0,\quad f(t, Tu) = Tf(t+\Delta t, u)
    \end{align*}
    then $\tau = k\Delta t$, as used in Theorem 5, \cite{russo2011symmetries}.
    \item (Nonlinear symmetry) A diffeomorphism $h \in C^1(V,V)$ generating a cyclic group of maps of order $k$ by the fixed point condition $(h \circ ... \circ h)(u) =: h^k(u) = u$ for all $u \in V$, and $\Delta t > 0$ such that:
    \begin{align*}
        \forall n \le k,\ u \in V,\ t \ge 0,\quad f(t, h^n(u)) = Dh^n(u)f(t + n\Delta t, u)
    \end{align*}
    where similarly $\tau = k\Delta t$.
\end{itemize}

\subsection{Limit cycle analysis}
We now combine the previous results on symmetries (Section \ref{sec:nonlinear_symmetries}) with criteria for convergence to submanifolds (Section \ref{sec:manifold_contraction}) to give a general criterion for convergence to a limit cycle in a Banach space. We proceed using surjectively weighted spaces as in Section \ref{sec:manifold_contraction} to relax from contraction to a fixed point to contraction to a \textit{loop}. We make use of diffeomorphisms to allow the loop shape to vary, and refer to these diffeomorphic coordinate changes as \textit{conjugate} dynamics.

\begin{theorem}[Contraction to a limit cycle]
\label{thm:limit_cycle}
Let $V$ be a Banach space and $\frac{d}{dt}u = f(t, u)$ a nonautonomous system in $V$. Suppose there exists:
\begin{enumerate}
    \item A smooth submersion $\phi\in C^{\infty}(V, \R)$ such that $\phi^{-1}(0) =: \hat{S}^1$ is an embedding of the unit circle $S^1$ in $V$
    \item A diffeomorphism $h : V \to V$ with $v = h(u)$
    \item A period $\tau > 0$
\end{enumerate}
such that, for all $t \ge 0, w \in \hat{S}^1$:
\begin{enumerate}
    \item (Loop invariance) $D\phi(w)\frac{dv}{dt}(t, w) = 0$
    \item (Loop contraction) $\sup_{t, x \in V}M_{\Theta(x)}\left[\frac{\partial}{\partial v}\frac{dv}{dt}(t, x)\right] < 0$, where $\Theta(x) = D\phi(x)$
    \item (Loop symmetry) $\frac{dv}{dt}(t, w) = \frac{dv}{dt}(t + \tau, w)$
    \item (Loop non-accumulation) $\frac{dv}{dt}(t, w) \ne 0$
\end{enumerate}
Then any solution $u(t)$ asymptotically converges to a (not necessarily synchronizing) limit cycle on $h^{-1}(\hat{S}^1)$
\end{theorem}
\begin{proof}
Taking the coordinate transform $v = h(u)$, $v$ has the dynamics:
\begin{align}
\label{eq:projected_cycle}
\frac{d}{dt}v = Dh(u)\frac{d}{dt}u = Dh(h^{-1}(v))f(t, h^{-1}(v)) =: g(t, v)
\end{align}
By hypotheses (1) and (2) and Theorem \ref{thm:banach_submanifold_contraction},  $v(t)$ satisfies the conditions for contraction to the submanifold $\hat{S}^1$:
\begin{equation}
\begin{split}
    \label{eq:loop_contraction}
    &1.\ \forall t \ge 0, z \in \hat{S}^1,\quad D\phi(z)g(t,z)  = 0\\
    &2.\ \sup_{t, v \in V} M_{\Theta(w)}[Dg_t(w)] < 0
\end{split}
\end{equation}
hence the conjugated system $v(t)$ exponentially approaches the unit circle $\hat{S}^1$. 

Next, hypothesis (3) implies that for any solution $v(t)$ of \eqref{eq:projected_cycle} with initial condition $v(0) \in \hat{S}^1$, we have $v_n(t) = v(t + \tau)$ and thus $v_n(0) = v(\tau)$ is also a solution for each $n \in \N$. However, this does not preclude solutions which are arc-wise periodic on $\hat{S}^1$. 

Finally, hypothesis (4) excludes accumulating trajectories on $\hat{S}^1$:
\begin{align}
    \label{eq:cycle_orientation}
    \forall t\in \R^+,\ v \in \hat{S}^1,\quad  g(t, v) \ne 0
\end{align}
which by continuity of $g$ implies that solutions of \eqref{eq:projected_cycle} cannot change orientation on $\hat{S}^1$.

Together, conditions (3) and (4) imply that solutions starting in $\hat{S}^1$ are $\tau$-periodic cycles on $\hat{S}^1$ (with possibly non-unit winding number); while the loop contraction conditions (1) and (2) imply that solutions $v(t)$ starting anywhere in $V$ approach the loop $\hat{S}^1$. 

Thus, the original system $u(t)$ asymptotically approaches a limit cycle on the (possibly non-circular) loop $h^{-1}(\hat{S}^1)$.
\end{proof}

We note that the above result is also novel in finite dimensions, to the authors' knowledge, and can be extended more generally to \textit{asymptotic} contractions to limit cycles on the loop $\hat{S}^1$ via the use of further weighted spaces as discussed following Theorem \ref{thm:banach_submanifold_contraction}.
Importantly, the above result does not require that solutions additionally converge to one another in phase-space, which would imply a synchronization property. 

Lastly, we combine the criteria for contraction to a limit cycle along with contraction to subspaces (Theorem \ref{thm:subspace}) to show phase-locking phenomena in \textit{heterogeneous} coupled-oscillator systems.

\begin{theorem}[Phase-locking in heterogeneous limit cycles]
\label{thm:period_sync}
Let $V$ be a Banach space, and $\phi$ and $\hat{S}^1 \subset V$ a loop submanifold as in Theorem \ref{thm:limit_cycle}. Suppose $(u_1, ..., u_i, ..., u_n) =: u \in V^n$ is the coupled system 
\begin{equation}
\begin{split}
    \label{eq:hetero_sys}
    \frac{du}{dt} &= f(t,u)
\end{split}
\end{equation}
Suppose there exist diffeomorphisms $h_i$ such that:
\begin{enumerate}
    \item (Limit-cycle leader) At least one of the conjugate dynamics $v_i = h_i(u_i)$, $v := (h_1(u_1), ..., h_n(u_n))$ is contracting to $\hat{S}^1$.
    \item (Asymptotically phase-locked followers) Angles $\theta_i$, corresponding rotation operators $R(\theta_i) : V \to V$ preserving $\hat{S}^1$, and the subspace $W \subset V^n$ consisting of rotations of some element $v \in V$:
    \begin{align*}
        W &= \{[R(\theta_1)v, ..., R(\theta_n)v]\ |\ v \in V\}
    \end{align*}
    and a corresponding projection operator $\im(P) = W,\ P^2 = P$ such that 
    the conjugate dynamics $v_i$ are contracting in the $(I-P)$-weighted space:
    \begin{align*}
        \sup_{t \ge 0,\ v \in \hat{S}^n} & M_{I-P}\left[\frac{dv}{dt}(t, v)\right] < 0
    \end{align*}
\end{enumerate}
Then each of the individual systems in \eqref{eq:hetero_sys} approaches a (possibly distinct) periodic attractor with common fixed period $T$.
\end{theorem}
\begin{proof} 
The proof consists of combining several prior results.
By hypothesis (1) and Theorem \ref{thm:limit_cycle}, at least one conjugate system $h_i(u_i(t))$ converges to a limit cycle on the loop $\hat{S}^1$. Moreover, by hypothesis (2) and Theorem \ref{thm:subspace}, the entire conjugate dynamics $(h_1(u_1),..., h_n(u_n))$ is contracting to the subspace of constant phase-shift on $\hat{S}^1$, and hence eventually achieves constant \textit{helicity} on the torus $\hat{T}^n := (\hat{S}^1)^n$. This implies that all the conjugate dynamics $v_i(t)$ approach phase-shifts of the same limit cycle on $\hat{S}^1$.

Thus, the original systems $u_i(t)$ asymptotically approach individual limit cycles $h_i^{-1}(\hat{S}^1)$ with the same period, resulting in ``phase-locking'' behavior.
\end{proof}

In the above result, the use of arbitrary homeomorphisms $h_i$ enabled the composition of rotations $r_i$ and time-dilations $d_i$ of $\hat{S}^1$ to restate the condition of phase-locking as the equivalent (simpler) condition that all systems approach exactly the same limit-cycle, and that it is constant-speed:
\begin{align}
    \label{eq:sync_cycle}
    \exists \tau > 0,\ \forall i,j,\ t\in \R^+,\ v \in \hat{S}^1,\quad g_j(t, v) = g_i(t,v) = g_i(t + \tau, v) 
\end{align}
The use of subspace contraction in condition (2) additionally allows the phase-locking property to be satisfied asymptotically, rather than uniformly in time.


\section{Application: long-time behavior of PDEs}
\label{sec:ap_pdes}

The classical contraction mapping principle is widely used to establish existence and well-posedness properties for nonlinear equations and PDEs \cite{evans1980solving}; however, while the contraction rate in a particular norm may be infinite, the asymptotic contraction rate in that norm may be finite (see Section \ref{sec:metric_dependence}). Thus, using results from Sections \ref{sec:main} and \ref{sec:symmetries}, we show that an ``asymptotic contraction mapping principle'' enjoys broader applicability in establishing existence, continuous data-dependence, and asymptotic stability. As before, we verify asymptotic contractions using weighted Banach spaces (see Definition \ref{eq:weighted_bsp}).

In the following, we consider a PDE of the form 
\begin{align}
    \label{eq:pde1}
    \partial_t u &= f(t, D^{\alpha^1} u, ..., D^{\alpha^n} u)
\end{align} 
with $f$ continuously differentiable, $u$ defined on a spatial domain $\Omega = \R^d$ or $\Omega = \R^d / \Z^d$, and $\{\alpha^i\}_{i=1}^n$ multi-indices up to some desired regularity $k$ (so $u \in C^k(\Omega)$ or $u \in W^{k, p}(\Omega)$ in a weak form). We use the mixed-partial derivative notation:
$$
D^{\alpha} = \frac{\partial^{|\alpha|}}{\partial x_1^{\alpha_1}...\partial x_d^{\alpha_d}},\quad |\alpha| = \sum_i\alpha_i
$$

\subsection{Regularity}
\label{sec:regularity}

We formally consider \eqref{eq:pde1} as an ODE $\partial_t u = F(t, u)$ with the nonlinear nonautonomous operator 
\begin{align}
    \label{eq:operator_ode}
    F(t, \cdot) = f(t, D^{\alpha^1}\cdot, ..., D^{\alpha^n}\cdot)
\end{align}
defined on the space of compactly supported smooth functions $C_c^\infty(\Omega)$. Let $M_{k, p}$ be the contraction rate in $C_c^\infty(\Omega)$ with respect to the Sobolev norm $W^{k,p}(\Omega)$. We use $M^{k, p}$ to indicate the the contraction rate induced by a semi-inner product on $W^{k, p}$. For example, if $p = 2$, 
\begin{align}
\label{eq:kp_rate}
M^{k, 2}[F_t] = \sup_{\phi_1\ne\phi_2 \in C_c^\infty(\Omega)}\frac{\sum_{i=0}^k\Re \langle \phi_1^{(j)} - \phi_2^{(j)}, F(t,\phi_1^{(j)}) - F(t,\phi_2^{(j)})\rangle}{\sum_{j=0}^k\langle \phi_1^{(j)} - \phi_2^{(j)}, \phi_1^{(j)} - \phi_2^{(j)}\rangle}
\end{align}
More generally, for $1 < p < \infty$, the norm is Gateaux-differentiable, with explicit formulas available for the semi-inner product (see \cite{zhang2009reproducing}). In $L^p$, formula \eqref{eq:kp_rate} is expressible as:
\begin{align*}
    [u, v]^{W^{k,p}} &= \sum_{j=0}^{k}[u^{(j)}, v^{(j)}]^{L^p}
\end{align*}

Suppose a smooth solution $\phi_*(t) \in C^1(\R^+, C_c^\infty(\Omega))$ to \eqref{eq:operator_ode} exists for all time.
If $F$ has a bounded $k, p$-asymptotic expansion rate, i.e. there exists $b > 0, \lambda_b \in \R$ and $\Theta(t, u)$ such that $\norm{\Theta(t, u)} < b$ uniformly and 
\begin{align}
\label{eq:kp_expansion}
\sup_{t \ge 0, \phi \in C_c^\infty(\Omega)} M_{\Theta(t, \phi)}^{k, p}[DF_t(\phi)] \le \lambda_b
\end{align}
then by the existence hypothesis and Theorem \ref{thm:c_cont}, any smooth initial condition $\phi \in C_c^\infty(\Omega)$ satisfies the following $W^{k, p}$-continuous dependence on initial conditions:
\begin{align}
\label{eq:regularity}
\norm{\phi(t) - \phi_*(t)}_{k,p} \le b^2e^{\lambda_b t}\norm{\phi(0) - \phi_*(0)}_{k, p}
\end{align}
Therefore, smooth initial conditions cannot lose their regularity in finite time. 

If furthermore, we define an extension of \eqref{eq:pde1} with $D^\alpha$ taken in the weak sense, i.e. $D^\alpha u := v$ where 
\begin{align*}
    \int_\Omega D^\alpha u \phi dx &= (-1)^{|\alpha|}\int_\Omega v D^\alpha \phi dx\quad \forall \phi \in C_c^\infty(\Omega)
\end{align*}
such that the corresponding vector field  has domain $\dom(F_t) = W^{k, p}(\Omega)$, and if $F$ satisfies the asymptotic expansion rate \eqref{eq:kp_expansion} instead the weak sense:
\begin{align}
\label{eq:kp_expansion_weak}
\sup_{t \ge 0, u \in W^{k, p}(\Omega)} M_{\Theta(t, u)}^{k, p}[DF_t(u)] \le \lambda_b
\end{align}
then weak solutions cannot lose their weak-differentiability in finite time. 

As a corollary, we can always ``improve'' regularity by adding a dissipative operator $G(t, u)$ such that there exists $b, \lambda_b, \Theta(t, u)$ as before satisfying a finite asymptotic expansion rate:
\begin{equation}
\begin{split}
    \label{eq:regularized}
\partial_t u &= f(t, D^{\alpha^1} u, ..., D^{\alpha^n} u) + G(t, u)\\
    F(t, \cdot) &= f(t, D^{\alpha^1}\cdot, ...)\\
    \sup_{t \ge 0, \phi \in C_c^\infty(\Omega)} &M_{\Theta(t, \phi)}^{k, p}[DF_t(\phi) + DG_t(\phi)] \le \lambda_b
\end{split}
\end{equation}
As we discuss in the Section \ref{sec:dissipative}, $G$ is often a Laplace operator in physical equations (see e.g., vanishing viscosity method \cite{evans1998partial}); however, \eqref{eq:regularized} significantly generalizes possible regularizing terms to time-varying dissipative effects in an arbitrarily weighted Sobolev space. (In this context, ``weighted Sobolev space'' refers to our construction of weighted Banach spaces (\ref{def:w_bsp}) rather than the typical definition $\norm{\langle \cdot\rangle^d f}_{L^p(\Omega)}$)

\subsection{Existence}

In the following, we describe a ``weighted contraction mapping principle'' for establishing the existence of solutions to PDEs and other nonlinear operator equations. In particular, we make use of weighted Banach spaces along with Theorem \ref{thm:c_cont} to verify that an operator is asymptotically norm-contracting, which along with an additional boundedness criterion, will be used to show existence of unique fixed points and limits of approximations.

\subsubsection{Existence for time-independent equations}
Consider a time-independent equation of the form
\begin{align}
    \label{eq:time_indep}
    f(D^{\alpha^1} u, ..., D^{\alpha^n} u) &= 0
\end{align}
for whom there exists $b > 0$ and $\lambda_b$ such that $M_{\Theta(u)}[f] \le \lambda_b < 0$ in some weighted Banach space $V_\Theta$ with $\norm{\Theta(t, u)} < b$ (see \ref{def:weighted_sip}). By Theorem \ref{thm:c_cont}, the time-dependent equation $\partial_tu = f(D^{\alpha^1} u, ..., D^{\alpha^n} u)$ is asymptotically contracting in the distance $\norm{\cdot}_V$, and therefore converges to a unique fixed point $u_*$. This $u_*$ is the unique solution to \eqref{eq:time_indep}. Importantly, we require only that the \textit{asymptotic} contraction rate under some invertible $b$-bounded weight $\Theta(u) \in GL(V, W)$ with codomain $W$ be negative.

\subsubsection{Existence of limits of regularized equations} 
\label{sec:dissipative}
When \eqref{eq:pde1} is difficult to solve directly or does not admit classical (sufficiently differentiable) solutions $u$ for all initial data, a common method is to consider the ``regularization'' of \eqref{eq:pde1} by some dissipative operator (see \eqref{eq:regularized}). In physical settings, this may correspond to damping, friction, or viscous-like effects. Such terms are chosen to enable the existence of smooth solutions, and we then consider a limit as the dissipating effect vanishes:
\begin{align}
    \label{eq:pde2}
    \partial_t u_\varepsilon &= f_\varepsilon(D^{\alpha^1} u_\varepsilon, ..., D^{\alpha^n} u_\varepsilon),\quad f_\varepsilon \to f \text{ as } \varepsilon \to 0^+
\end{align}
with convergence on $f_\varepsilon$ appropriately strong. The goal is to pass from a limit in $\varepsilon$ to a limit in $u_\varepsilon$.

Using a compactness argument, we show the existence of such limits when the regularized equation \eqref{eq:pde2} has a uniformly bounded contraction rate in some weighted Banach space.

\begin{theorem}[Approximation family with uniform $L^p$ contraction rate has a convergent subsequence]
\label{thm:limit1}
Let $u_0 \in L^p(\T^n)$ where $\T^n$ is the $n$-torus $\R^n/\Z^n$.
Suppose there exist $b, c > 0$ and $\lambda \in \R$, such that for all $\varepsilon > 0$, the regularized Cauchy problem \eqref{eq:pde2} admits solutions $\{u_\varepsilon\}_{\varepsilon>0}$ defined on $[0, T]\times\T^n$ satisfying $u_\varepsilon(0,x) = u_0(x)$, and:
\begin{enumerate}
    \item (Uniform $L^p$ asymptotic contraction rate) There exists $\Theta_\varepsilon(t, p)$ such that $M_{\Theta_\varepsilon}[f_\varepsilon] = \lambda_\varepsilon \le \lambda$ and $\sup_{t, p}\max(\norm{\Theta_\varepsilon(t,p)}, \norm{\Theta_\varepsilon(t, p)}^{-1}) \le b$
    \item (Existence of an $L^p$ bounded solution) There exists a solution $u_\varepsilon^*(t,x)$ (with arbitrary initial data) satisfying $\sup_{t \in [0, T]}\norm{u_\varepsilon^*(t,\cdot)}_p \le c$
    \item (Translation invariance) For all $y \in \R^n$ and torus translations $\tau_y :\T^n\to\T^n$, the function $u_\varepsilon(t, \tau_y x)$ is a solution with $u_\varepsilon(0, \tau_yx) = u_0(\tau_yx)$
    \item (Uniform continuity about initial data)  $\lim_{t\to0^+}\norm{u_\varepsilon(t,\cdot) - u_0}_{L^p(\T^n)} = 0$ uniformly on $\{u_\varepsilon\}$
\end{enumerate}

Then the sequence $\{u_\varepsilon\}$ with $\varepsilon\to 0^+$ has a convergent subsequence $\{u_{\varepsilon_j}\}_{j=1}^\infty$ with limit $u \in L^p([0, T] \times \T^n)$.
\end{theorem}
\begin{proof} 
Let $\varepsilon > 0$.
We first show that $\{u_\varepsilon\} \subset L^p([0, T]\times \T^n)$ and it is bounded. Applying hypotheses (1), (2), and Theorem \ref{thm:c_cont}, there exists $u_\varepsilon^*(t,x)$ such that:
\begin{align*}
    \norm{u_\varepsilon}_{L^p([0, T]\times \T^n)}^p &\le  \int_0^T \norm{u_\varepsilon(t,\cdot)}_{L^p(\T^n)}^pdt\\ 
    &\le T\sup_{t \in [0, T]}\norm{u_\varepsilon^*(t,\cdot)}^p + \int_0^T \norm{u_\varepsilon(t,\cdot) - u_\varepsilon^*(t,\cdot)}^pdt\\
    &\le Tc^p + \frac{b^2}{\lambda_\varepsilon p}(e^{\lambda_\varepsilon pT} - 1)\left(\norm{u_0}^p + \norm{u_\varepsilon^*(0,\cdot)}^p\right)\\
    &\le Tc^p + B(0, T)\left(\norm{u_0}^p + c^p\right)< \infty
\end{align*}
where we have defined the function:
\begin{align*}
    B(s, t) = \frac{b^2}{\lambda p}(e^{\lambda pt} - e^{\lambda ps})\quad \text{ for } s \le t
\end{align*}
and the $\varepsilon$-uniform bound on the final line uses $\lambda_\varepsilon < \lambda$ from hypothesis (1).

Next, we define the extension of $u_\varepsilon$ to $\hat{u}_\varepsilon \in L^p(\R^{n+1})$ as follows:
\begin{align*}
    \hat{u}_\varepsilon(t, x) &= \begin{cases} u_\varepsilon(t, x) & t \in [0, T],\ x \in [0,1)^n\\ 0 & \text{otherwise}\end{cases}
\end{align*}
We now show that $\{\hat{u}_\varepsilon\}$ is equicontinuous. Let $0 < \delta < T$ and $h \in \R, y \in \R^n$ such that $|h|, |y| < \delta$; then:
\begin{align*}
    \norm{\hat{u}_\varepsilon(\cdot-h, \tau_y\cdot) - \hat{u}_\varepsilon}^p &\le \norm{\hat{u}_\varepsilon(\cdot-h, \tau_y\cdot) - \hat{u}_\varepsilon(\cdot-h,\cdot)}^p + \norm{\hat{u}_\varepsilon(\cdot-h,\cdot) - \hat{u}_\varepsilon}^p
\end{align*}
Applying hypotheses (1) and (3) to the first term,
\begin{align*}
    \norm{\hat{u}_\varepsilon(\cdot-h, \tau_y\cdot) - \hat{u}_\varepsilon(\cdot-h,\cdot)}^p &\le \int_{0}^T\norm{u_\varepsilon(t, \tau_y\cdot) - u_\varepsilon(t,\cdot)}^p_{L^p(\T^n)}dt \\
    &\le B(0, T) \norm{u_0(\tau_y \cdot)- u_0}^p
\end{align*}
Since $u_0 \in L^p(\T^n)$, we have $\norm{u_0(\tau_y\cdot)- u_0}_{L^p(\T^n)}^p \to 0$ as $\delta \to 0^+$ by density of $C_c^\infty \subset L^p$ (Lemma 4.3, \cite{brezis2011functional}). Applying hypotheses (1) and (2) to the second term and using the assumption $|h| < \delta < T$, 
\begin{align*}
    \norm{\hat{u}_\varepsilon(\cdot-h,\cdot) - \hat{u}_\varepsilon}_{L^p(\R^{n+1})}^p &\le \int_{T - |h|}^T\norm{u_\varepsilon(t, \cdot)}^pdt + \int_0^{|h|}\norm{u_\varepsilon(t, \cdot)}^pdt \\
    &+ \int_{|h|}^T\norm{u_\varepsilon(t-|h|,\cdot)-u_\varepsilon(t,\cdot)}^pdt\\
    &\le 2|h| c^p + (B(T-|h|,T)+B(0,|h|))(\norm{u_0}^p + c^p) \\
    &+B(|h|,T)\norm{u_\varepsilon(|h|,\cdot)-u_0}^p
\end{align*}
Since $B(s, t) \to 0$ as $s \to t$, the first two terms approach $0$ uniformly on $\{u_\varepsilon\}$ as $\delta\to 0^+$. Hypothesis (4) gives the uniform convergence in the final term.

Lastly, since $\hat{u}_\varepsilon$ is supported only on $[0, T]\times[0, 1)^n$ by definition, there exists $R > 0$ such that 
\begin{align*}
    \int_{|t|,|x| > R}|\hat{u}_\varepsilon(t,x)|^p &= 0
\end{align*}
hence $\{\hat{u}_\varepsilon\}$ is equitight.

By the Kolmogorov-Riesz(-Fr\'echet) compactness theorem (Theorem 5, \cite{hanche2010kolmogorov} or Corollary 4.27, \cite{brezis2011functional}), the boundedness, equicontinuity, and equitightness of $\{\hat{u}_\varepsilon\}$ in $L^p(\R^{n+1})$ implies that $\{u_\varepsilon\}$ is relatively compact in $L^p([0, T] \times\T^{n+1})$ by considering the restriction to $[0, T] \times [0,1)^n$. Hence, for any sequence $\{\varepsilon_k\}_{k=1}^\infty$ with $\varepsilon_k \to 0^+$, there exists a convergent subsequence $\{u_{\varepsilon_j}\}_{j=1}^\infty$ with $u_{\varepsilon_j}\to u \in L^p$.
\end{proof}

Theorem \ref{thm:limit1} enables one to pass to a subsequential limit of $u_\varepsilon$ in a sequence of regularized equations \eqref{eq:pde2}, by transferring the (uniformly bounded) contraction rate of a family of approximate solutions to an accumulation point of the family. Subsequently, one verifies that any such accumulation point solves the original equation with $\varepsilon = 0$ in the weak sense.

This approach is inspired by the vanishing viscosity method \cite{crandall1992user} used to construct weak solutions to a variety of first-order equations, such as nonlinear conservation laws \cite{bianchini2005vanishing, lattanzio2003global, kruzhkov1970first}. In this case, the dissipative effect of the Laplace operator enables both the existence of classical solutions and the uniform control of continuity required to show compactness.

On the other hand, Theorem \ref{thm:limit1} may be viewed as a generalization of this method (separate from the generalization to \textit{viscosity solutions} \cite{crandall1992user}) by allowing approximations which are expansive (i.e., $\lambda_\varepsilon$ are allowed to be positive) in an arbitrarily weighted $L^p$ norm as in Definition \ref{eq:weighted_rate}; thus, the approximation $f_\varepsilon$ need not arise from purely dissipative effects. Due to the norm- and weight-dependence of contraction rates described in Section \ref{sec:metric_dependence}, this substantially increases the set of possible regularized approximations one may consider in \eqref{eq:pde2}. We refer to this as a vanishing \textit{one-sided Lipschitz} approximation, since in the unweighted norm, the contraction rate $M_V[f_\varepsilon]$ (Definition \ref{def:weighted_rate}) upper-bounds the one-sided Lipschitz constant of $f$ (see \cite{aminzare2014contraction}), while in more general weighted norms imposes a uniform control on the expansion rate of the dynamics.

\subsection{Examples}

Exponential convergence of the heat equation to $0$ is readily verified in the space of trace-zero functions using unweighted contraction or Fourier analysis (Proposition 6.1, \cite{soderlind2006logarithmic}). On the other hand, in the case of conservative (zero-flux) boundaries, we can apply contraction in surjectively weighted spaces using Theorem \ref{thm:subspace} to establish exponential convergence to the subspace of constant functions without Fourier analysis.

\begin{example}[Heat equation with zero-flux boundaries]
Let $u : \Omega \to \R$ with $\Omega \subset \R^d$ a nonempty, open, connected, and bounded set with boundary $\partial \Omega$ of class $C^1$, and let:
\begin{equation}
\begin{split}
    \label{eq:heat_zf}
     u_t &= \alpha\Delta u\\
    \nabla u \cdot \hat{n} &= 0
\end{split}
\end{equation}
where $\hat{n}$ is the boundary normal. Furthermore, let 
\begin{align}
    \label{eq:c_proj}
    Pu &= \frac{1}{|\Omega|}\int_\Omega u
\end{align}
be the $L^2$-orthogonal projection onto the subspace of constant functions $W = \{v \in L^2(\Omega)\ |\ v \equiv c \in \R \}$. By the zero-flux conditions, $Q\Delta Pu = Q0 = 0$, hence $W$ is an invariant subspace. Letting $Q = I - P$, we have 
\begin{align}
    \label{eq:weighted_laplacian}
    M_Q[\Delta] &= \sup_{v \in L^2(\Omega)}\frac{\langle Qv, Q\Delta v\rangle}{\norm{Qv}^2} = \sup_{v \in W^\perp}\frac{\langle \nabla v, \nabla v\rangle}{\norm{v}^2} < 0
\end{align}
thus by Theorem \ref{thm:subspace}, \eqref{eq:heat_zf} exponentially converges to the subspace of constant functions. The last is by the Poincar\'e-Wirtinger inequality.
\end{example}

\begin{example}[Existence and uniqueness for a nonlinear Poisson equation]
Let $u : \Omega \to \R^n$ with $\Omega = [0, 1]^d$, and consider the second-order equation:
\begin{equation}
\begin{split}
    \label{eq:poisson_like}
    L u + f(u) &= 0\\
    u &= 0 \text{ on } \partial \Omega
\end{split}
\end{equation}
with $f \in C^1(\R^n, \R^n)$ and $L$ defined as the continuous linear extension of the Laplacian densely $\Delta$ densely defined on $C_c^\infty(\Omega) \subset W_0^{1,2} = V$ (the Sobolev space $W^{1,2}$ with vanishing Dirichlet conditions). By the divergence theorem,
\begin{equation}
\begin{split}
    \label{eq:laplace_neumann}
    M(\Delta) &= \sup_{\phi \ne 0}\frac{\langle \phi, \nabla \cdot \nabla \phi\rangle}{\langle \phi, \phi\rangle} = \sup_{\phi \ne 0}\frac{\int_{\Omega}(\nabla \cdot (\phi\nabla \phi) - \nabla \phi \cdot \nabla \phi)d\Omega}{\langle \phi, \phi\rangle} \\
    &= \sup_{\phi \ne 0}\frac{\int_{\partial \Omega}\phi\nabla \phi \cdot \hat{n}\partial \Omega - \langle \nabla \phi, \nabla \phi\rangle }{\langle \phi, \phi\rangle} = \sup_{\phi \ne 0}\frac{-\langle\nabla \phi, \nabla \phi\rangle}{\langle \phi, \phi\rangle} \le -\lambda(\Omega)
\end{split}
\end{equation}
for some $\lambda(\Omega) > 0$ by the Poincar\'e inequality \cite{evans1998partial}. Furthermore, if there exists any uniformly bounded family of weights $\Theta(u) : V \to GL(V)$ such that:
\begin{align*}
    \sup_{u \in V} M_{\Theta(u)}[Df(u)] &= \sup_{u \in V, v \ne 0}\frac{\langle\Theta(u)v, \Theta(u)Df(u)v\rangle}{\norm{v}^2} < \lambda(\Omega)
\end{align*}
Then \eqref{eq:poisson_like} is an asymptotic contraction in $V$ by Theorem \ref{thm:c_cont}, the time-dependent dynamics (letting $0 \mapsto u_t$) converges to a unique fixed point  (Theorem 2.1, \cite{kirk2003fixed}), and there \eqref{eq:poisson_like} a unique solution. We note that the standard uniqueness proof for the linear inhomogeneous equation $\Delta u + f = 0$ by subtracting the difference of solutions does not apply, due to presence of the ``feedback'' term $f(u)$.
\end{example}

Prior work has shown that a reaction term with sufficiently bounded expansion rate in unweighted (6.6, \cite{soderlind2006logarithmic}) and diagonally weighted \cite{aminzare2013logarithmic} norms is sufficient to guarantee spatially homogeneous equilibria in reaction-diffusion systems. We use weighted contractions to give the most general version of this condition as a bounded expansion rate in the subspace of non-constant functions.

\begin{example}[Suppressing pattern formation in reaction-diffusion systems]
Let $\Omega$ be a domain as in \eqref{eq:heat_zf} and a nonautonomous system of coupled reaction-diffusion systems on $\Omega$ with zero-flux boundary conditions:
\begin{equation}
\begin{split}
    \label{eq:rd}
    \partial_t u_i &= \alpha_i\Delta u_i + f_i(t, u_1, ..., u_n),\quad \alpha_i > 0\\
    \nabla u_i \cdot \hat{n} & = 0
\end{split}
\end{equation}
and let $P$ be a projection onto the subspace of constant functions as in \eqref{eq:c_proj}. Letting $Q = I-P$, in \eqref{eq:weighted_laplacian} we showed that the weighted contraction of rate of the Laplacian is $M_Q(\Delta) < 0$ in $H^{2}(\Omega)$. Consequently, if 
\begin{enumerate}
    \item $Qf_i(t, Pv) = 0$ for all $i$, $t \ge 0$, $v \in H^{2}(\Omega)$
    \item $M_Q[f_i] < \alpha_i|M_Q[\Delta]|$ for all $i$
\end{enumerate}
Then by Theorem \ref{thm:subspace} and sub-additivity of the weighted contraction rate, \eqref{eq:rd} exponentially converges to a spatially homogeneous solution.
\end{example}




\section{Conclusion}

In this work, we introduced weighted Banach spaces \eqref{eq:weighted_bsp} as a means of verifying a broad range of asymptotically incrementally stable behaviors of infinite-dimensional dynamical systems. These are families of topologically equivalent normed spaces, emulating the use of Riemannian metrics in classical (geodesic) contraction analysis \cite{lohmiller1998contraction}, by endowing tangent spaces with continuously time- and space-varying weighted semi-inner products. The weighted contraction rate of an operator was formulated using the numerical range induced by this semi-inner product \eqref{eq:weighted_rate}. In the case of uniformly bounded invertible weights, this enabled the verification of asymptotically norm-contracting differential equations under a much wider range of conditions than the classical contraction mapping principle, and gave rise to a natural invariant, the asymptotic contraction rate \eqref{eq:as_rate}, formulated by taking the best-case contraction rates over the choice of such weights. The asymptotic contraction rate and its time-uniform counterpart were shown to upper-bound the maximum Lyapunov exponent (Section \ref{sec:metric_dependence}), while remaining invariant under a broader family of coordinate-transforms and generalizing well to the non-autonomous, infinite-dimensional case.

In the second part of this work, we developed the notion of weighted contractions further by introducing surjectively weighted spaces, which were used in a number of ways to verify non-equilibrium asymptotic behaviors such as convergence to subspaces (\ref{thm:subspace}), submanifolds (\ref{thm:banach_submanifold_contraction}), group-action invariant solutions (\ref{sec:symmetries}), limit cycles (\ref{thm:limit_cycle}), and phase-locking behaviors (\ref{thm:period_sync}). We developed several applications to the well-posedness and long-time analysis of PDEs (\ref{sec:ap_pdes}), notably the use of finite asymptotic expansion rates to develop continuous data-dependence properties (\ref{sec:regularity}) and an asymptotic contraction mapping principle for establishing limits of vanishing one-sided Lipschitz approximations (\ref{thm:limit1}). 

The use of weighted spaces extends the classical contraction mapping principle to either an asymptotic one (Theorem \ref{thm:c_cont} in combination with \cite{kirk2003fixed}) or one which verifies nonequilibrium limiting properties. In the former case, optimization over weights suggests a natural formulation for input synthesis of PDE controls, similarly to the method of control contraction metrics \cite{manchester2017control} in $\R^n$. Such procedures may be practically implemented using Galerkin methods. In the latter case, weighted contraction analysis in the infinite-dimensional setting provides a path to studying asymptotic non-equilibrium behaviors of many-body systems. In particular, passing contraction properties of a microscopic ODE description through continuum limits may yield an effective means of order reduction, especially for systems which may not be well-approximated by mean-field or hydrodynamic descriptions, but possess some invariant attractive manifold. Lastly, weighted contraction analysis for norms other than $L^p$ -- such as Wasserstein norms \cite{piccoli2019wasserstein} -- may have practical applications in stability analysis of physical processes such as transport.

\section*{Acknowledgements}    
This research did not receive any specific grant from funding agencies in the public, commercial, or not-for-profit sectors.

\bibliographystyle{elsarticle-num-names} 
\bibliography{references}

\end{document}